\documentclass{jams-l}

\usepackage{epsf}
\usepackage{epsfig}
\usepackage{graphicx}
\usepackage{xcolor}
\usepackage{bm}
\usepackage{tabularx}
\usepackage{amstext}
\usepackage{float}
\usepackage{array, nicefrac, caption}

\newcommand{\divprop}{\lfloor}
\newcommand{\propdiv}{\lfloor}
\newtheorem{lemma}{Lemma}
\newtheorem{theorem}{Theorem}
\newtheorem{corollary}{Corollary}
\newtheorem{definition}{Definition}
\newtheorem{question}{Question}

\newtheorem{conjecture}{Conjecture}

\newlength\tmp

\allowdisplaybreaks
\begin{document}

\frenchspacing

\title{Recursively abundant and recursively perfect numbers}

\author{Thomas Fink}

\date{\today}




\maketitle

\vspace{-0.2in}
\begin{center}
\begin{small}
London Institute for Mathematical Sciences, 35a South St, London W1K 2XF, UK \\
Centre National de la Recherche Scientifique, Paris, France
\begin{minipage}{0.87\textwidth}
\vspace{0.2in}
\noindent
The divisor function $\sigma(n)$ sums the divisors of $n$.
We call $n$ abundant when $\sigma(n) - n > n$ and perfect when $\sigma(n) - n = n$.
I recently introduced the recursive divisor function $a(n)$, the recursive analog of the divisor function.
It measures the extent to which a number is highly divisible into parts, such that the parts are highly divisible into subparts, so on.
Just as the divisor function motivates the abundant and perfect numbers, the recursive divisor function motivates their recursive analogs, which I introduce here.
A number is recursively abundant, or ample, if $a(n) > n$ and recursively perfect, or pristine, if $a(n) = n$.
There are striking parallels between abundant and perfect numbers and their recursive counterparts.
The product of two ample numbers is ample, and 
ample numbers are either abundant or odd perfect numbers.
Odd ample numbers exist but are rare, and I conjecture that there are such numbers not divisible by the first $k$ primes---which is known to be true for the abundant numbers.
There are infinitely many pristine numbers, but that they cannot be odd, apart from 1.
Pristine numbers are the product of a power of two and odd prime solutions to certain Diophantine equations, 
reminiscent of how perfect numbers are the product of a power of two and a Mersenne prime.
The parallels between these kinds of numbers hint at deeper links between the divisor function and its recursive analog, worthy of further investigation.
\end{minipage}
\end{small}
\end{center}

\vspace{0.3in}

\section{Introduction} 
\subsection{Parts into parts into parts}
A man wants to leave a pot of gold coins to his future children.
The coins are to be equally split among his children, even though he doesn't yet know how many children he will have.
The coins cannot be subdivided.
How many coins should the man leave?
On the one hand, he wants a number with many divisors, so that each of his children can get the same share, however many children he has.
On the other hand, he doesn't want a number so big that that he can't afford the coins. 
Suitable numbers for this problem are the highly composite numbers \cite{RamanujanA}, which have more divisors than any number preceding them.
\\ \indent
But now suppose that, in addition, the man wants each of his children to be able to split their share equally among their own children.
In this case the number of coins should be highly divisible, but also the parts into which it is divided should be highly divisible, too.
\\ \indent
This process can be extended down to great-grandchildren, and so on. 
I call this property---the parts are divisible into subparts which are divisible into sub-subparts, and so on---recursive divisibility.
To identify numbers that are recursively divisible to a high degree, I recently introduced and studied the recursive divisor function---the recursive analog of the usual divisor function \cite{Fink}.
I showed that the number of recursive divisors is twice the number of ordered factorizations into integers greater than one \cite{Fink}.
This latter problem has been well-studied in its own right by Kalmar, Hille, Erd\"{o}s, Chor, Klazara and Deleglise
\cite{Kalmar, Hille, Erdos, Chor, Klazara, Deleglise}.
\subsection{Modular design}
My original inspiration for introducing the recursive divisor function was, surprisingly, a problem in graphic design concerning the London Institute website.
In graphic design, a grid system is often used to divide the page into equal primitive parts \cite{Brockmann}.
These parts form the smallest unit from which larger parts can be composed, similar to how Lego bricks must be multiples of the minimal $1 \times 1$ brick.
The two dimensions can be treated separately, so consider just the width.
Into how many primitive columns should the page be divided?
This number, the grid size, should provide many options for dividing the page into equal columns, while being as small as possible to encourage simplicity and consistency.
\\ \indent
But now suppose that the columns themselves will be broken into sub-columns, and the sub-columns into sub-sub-columns, and so on.
This is a familiar task in traditional print media, such as newspapers. 
But with the advent of digital design, hierarchical modularity is becoming the rule, not the exception.
By computing those numbers which are more recursively divisible than all of their predecessors, 
we recover many of the grid sizes commonly used in design and technology, such as website grids and digital display resolutions \cite{Fink}.
Up to 1000, these are 1, 2, 4, 6, 8, 12, 24, 36, 48, 72, 96, 120, 144, 192, 244, 288, 360, 480, 576, 720, 864 and 960.
On the basis of this, the website of the London Institute uses a grid of 96 primitive columns.
\subsection{Outline of paper} 
Including this introduction, this paper is divided into four parts.
In part 2, I first review the recursive divisor function introduced previously, which has a geometric interpretation in the form of divisor trees \cite{Fink}.
I then show that the number of recursive divisors $a(n)$ is at least multiplicative: $a(l n) \geq a(l) \, a(n)$.
Just as the divisor function gives rise to the abundant, perfect and deficient numbers,
the recursive divisor function gives rise to their recursive counterparts.
I introduce ample numbers, for which $a(n) > n$;
pristine numbers, for which $a(n) = n$; and
depleted numbers, for which $a(n) < n$.
For each kind of number, an example divisor tree is shown in Figure \ref{threetypesplot}.
\\ \indent
In part 3 I investigate ample numbers, the recursive analog of abundant numbers.
Ample and abundant numbers have some curious parallel properties.
I show that the product of two ample numbers is ample, whereas any multiple of an abundant number is abundant.
Ample numbers, which are rarer than abundant numbers, are either abundant or odd perfect numbers (if they exist).
The first 100 are shown in Table 2.
The first $10^8$ ample numbers are even but, to my surprise, odd ample numbers exist. 
This is analogous to the abundant numbers, where the first odd abundant number is preceded by many even ones.
I conjecture that there exist ample numbers not divisible by the first $k$ primes, which is known to be true for abundant numbers \cite{Iannucci}.
I give the smallest such ample numbers for $k=1$ and $k=2$, which are approximately $10^{12}$ and $10^{87}$.
\begin{figure}[b!]
\raggedright
\setlength\tmp{1.2\textwidth}
			\includegraphics[width=\textwidth]{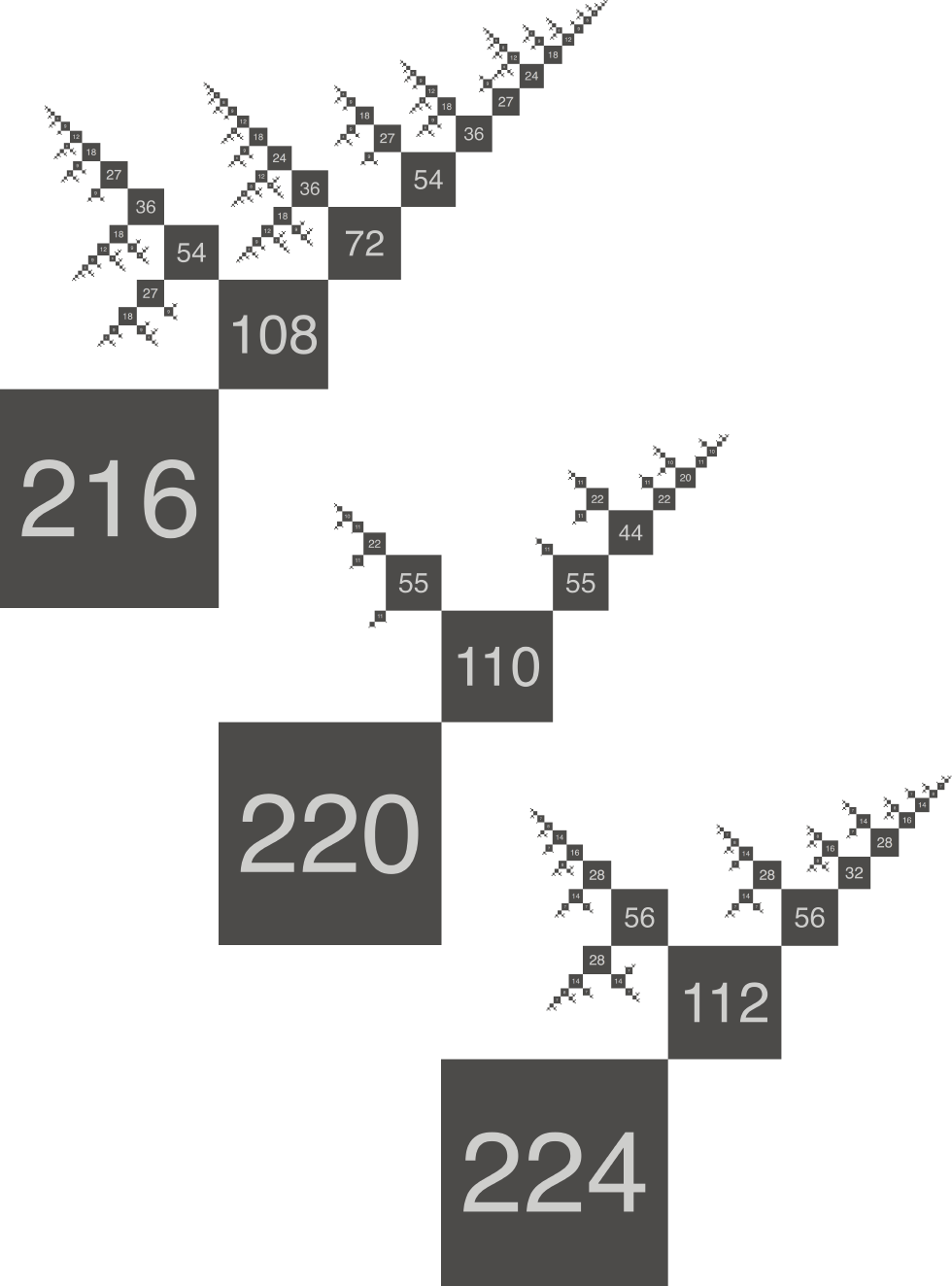} 
\caption{\small
Divisor trees for 216, 220 and 224. 
The number of recursive divisors $a(n)$ counts the number of squares in each tree.
For example, the proper divisors of 216 are 108, 72, 54 and so on, which form the main arm of its tree.
The proper divisors of 108 are 54, 36, 27 and so on, which form the first sub-arm of the tree.
From top, these numbers are examples of 
ample 	numbers ($a(n)>n$),
depleted 	numbers ($a(n)<n$) and 
pristine 	numbers ($a(n)=n$).
The number 216 is ample because $a(216) = 504 >216$;
220 is depleted because $a(220) = 88 < 220$; and
224 is pristine because $a(224) = 224$.
While 224 is the 11th pristine number, there are only 100 less than $8 \times 10^{18}$.
Divisor trees can be generated for any number $n$ at lims.ac.uk/recursively-divisible-numbers.
}
\vspace{-14pt}
\label{threetypesplot}
\end{figure}
\clearpage
%
In part 4 I investigate pristine numbers, the recursive analog of perfect numbers.
Pristine and perfect numbers also possess some parallel properties.
It is thought unlikely that odd perfect numbers exist, and I show that there can be no odd pristine numbers, apart from 1. 
Pristine numbers take the form of a product of a power of 2 and odd prime solutions to certain Diophantine equations,
whereas perfect numbers are the product of a power of two and a Mersenne prime.
All numbers of the form $2^c$ and $2^{q-2} q$ are pristine, where $q$ is an odd prime.
However, no numbers of the form $2^c q^3$ or $2^c q^5$ exist, and there can be at most $d-1$ pristine numbers of the form for $2^c q^d$ for $d$ odd.
Pristine numbers are never perfect apart from 6.
The first 100 pristine numbers are given in Table 2, the largest of which is approximately $10^{19}$.
Additionally, I compute the first several pristine numbers of the form
$2^c q \, r$, 		
$2^c q^2$, 		
$2^c q \, r \, s$ and 		
$2^c q^2 r$,
shown in Table 3.
\\ \indent
I conclude with a list of open problems.
\section{Number of recursive divisors}
\noindent
Throughout this paper I write $m|n$ to indicate $m$ divides $n$ and $m\divprop n$ to indicate $m$ is a proper divisor of $n$.
\\ \indent 
The usual divisor function,
\begin{equation*}
\sigma_x(n) = \sum_{m|n} m^x,
\end{equation*}
sums the divisors of $n$ raised to some integer power $x$.
When $x=1$, the divisor function sums the divisors of $n$ and is generally written $\sigma(n)$.
In this paper, however, I am concerned not only with the proper divisors of $n$ but also the proper divisors of its proper divisors, 
the proper divisors of the proper divisors of its proper divisors, and so on.
Recently I introduced and studied the recursive divisor function \cite{Fink},
\begin{equation*}
\kappa_x (n) = n^x + \sum_{m\divprop n} \kappa_x(m).
\end{equation*}
When $x=0$, I call this the number of recursive divisors $a(n)$.
\begin{definition}
	The number of recursive divisors is $a(1)=1$ and
	\begin{equation*}
		a(n) = 1 + \sum_{m\divprop n}a(m),
	\end{equation*}
	where $m\divprop n$ means $m$ is a proper divisor of $n$.
	\label{adef}
\end{definition}
\noindent
For example, $a(10) = 1 + a(1) + a(2) + a(5) = 6$.
Note that $a(n)$ depends only on the set of exponents in the prime factorization of $n$ and not on the primes themselves.
\subsection{Divisor trees}
The number of recursive divisors $a(n)$ has a geometric interpretation: it is the number of squares in the divisor tree of $n$.
Fig. \ref{threetypesplot} shows divisor trees for 216, 220 and 224.
By contrast, $\sigma(n)$ adds up the side lengths of the squares in the main diagonal of the trees.
Divisor trees can be generated for any number $n$ at lims.ac.uk/recursively-divisible-numbers.
\\ \indent
A divisor tree is constructed as follows.
First, draw a square of side length $n$.
Let $m_1, m_2, \ldots$ be the proper divisors of $n$ in descending order. 
Then draw squares of side length $m_1$, $m_2, \ldots$ with each consecutive square situated to the upper right of its predecessor.
This forms the main arm of a divisor tree.
Now, for each of the squares of side length $m_1, m_2, \ldots$, repeat the process.
Let $l_1, l_2, \ldots$ be the proper divisors of $m_1$ in descending order. 
Then draw squares of side length $l_1$, $l_2, \ldots$, but with the sub-arm rotated $90^{\circ}$ counter-clockwise.
Do the same for each of the remaining squares in the main arm. 
This forms the branches off of the main arm.
Continue this process, drawing arms off of arms off of arms, and so on, until the arms are single squares of size 1.
\\ \indent
The number of recursive divisors can be written in closed from for one, two and three primes to powers, as shown in \cite{Fink,Hille,Chor}.
Let $p, q$ and $r$ be prime. Then
\begin{eqnarray}
	a(p^c) 			&=& 		2^c 													,	\label{aexpc}	\\
	a(p^c q^d) 		&=& 		2^c \sum_{i=0}^d \binom{d}{i} \binom{c+i}{i}					,	\label{aexpcd}	\\
	a(p^c q^d r^e) 		&=& 		\sum_{j=0}^d (-1)^j \binom{d}{j} \binom{c+d-j}{d} a(p^{c+d-j} r^e)	.	\label{aexpcde}
\end{eqnarray}
\begin{theorem}
	For any two integers $l$ and $n$, $a(l n) \geq a(l) \, a(n)$.
	\label{aproductrule}
\end{theorem}
\emph{Proof.}
The proof is by induction on $l$.
First note that $a(l n) \geq a(l) \, a(n)$ for $l=1$ and all $n$.
Assume $a(k n) \geq a(k) \, a(n)$ for all $k < l$ and all $n$.
From Definition \ref{adef},
\begin{equation*}
a(l n) = 1 + \!\! \sum_{m\propdiv l n} \!\! a(m).
\end{equation*}
Let $t_1, t_2,\ldots, t_j$ be the proper divisors of $l$.
Then
\begin{eqnarray*}
a(l n) 	&\geq&	1 + a(t_1 n) + a(t_2 n) + \ldots + a(t_j n) + \sum_{m\propdiv n} a(m) 		\\
		&\geq&	a(t_1) \, a(n) + a(t_2) \, a(n) + \ldots + a(t_j) \, a(n) + a(n) 				\\
		&=&		\Big(1 + a(t_1) + a(t_2) + \ldots + a(t_j) \Big) \, a(n) 					\\
		&=&		\Bigg(1 + \sum_{m\propdiv l} a(m) \Bigg)	a(n) 					\\
		&=&		a(l) \, a(n),															
\end{eqnarray*}
completing the inductive step. \qed	
\section{Recursively abundant numbers} 
\noindent
In this section I review abundant numbers and introduce recursively abundant numbers, which I call ample numbers.
\subsection{Abundant numbers} 
A number is abundant if the sum of its proper divisors exceeds it, that is,
\begin{equation*}
\sigma(n) - n = \sum_{m \divprop n} m > n.
\end{equation*}
A number is deficient if the sum of its proper divisors is less than it, that is, $\sigma(n) - n < n$.
The first several abundant numbers are 12, 18, 20, 24, 30, 36, 40, 42.
All multiples of abundant numbers are abundant, so abundant numbers are not rare: their natural density is between 0.2474 and 0.2480 \cite{DelegliseB}.
\\ \indent
While the first 231 abundant numbers are even, the 232nd is odd: it is not divisible by the first prime.
In fact, there exist abundant numbers not divisible by the first $k$ primes, for all $k$ \cite{Iannucci}.
The smallest such numbers for the first few $k$ (OEIS A047802 \cite{Sloane}) are
\begin{eqnarray*}
945				&=&		3^3 \cdot 5 \cdot 7, 								 \\
5391411025 		&=&		5^2 \cdot 7 \cdot 11 \cdot \ldots \cdot 29, 				 \\
2.0 \times 10^{25}	&=&		7^2 \cdot 11^2 \cdot 13 \cdot 17 \cdot \ldots \cdot 67.
\end{eqnarray*}
\subsection{Ample numbers} 
Recursively abundant numbers are the recursive analog of abundant numbers.
I call them ample numbers.
\begin{definition}
A number $n$ is ample if $a(n) > n$ and depleted if $a(n) < n$, where $a(n)$ is the number of recursive divisors.
\end{definition}
\noindent
For example, 12 is ample because $a(12) = 16 > 12$, but 14 is depleted because $a(14) = 6 < 14$.
The first several ample numbers are 12, 24, 36, 48, 60, 72, 80, 84, and the first 100 are shown in Table 1.
Unlike abundant numbers, ample numbers become scarcer with  $n$, with the density apparently vanishing (Figure 2).
\\ \indent
One of the original motivations for studying $a(n)$ was identifying numbers which are recursively divisible to a high degree.
The record holders are the recursively highly composite numbers, which have more recursive divisors than all of their predecessors \cite{Fink}.
A less stringent benchmark is being ample; in terms of being recursively divisible, they just pass muster.
\begin{corollary}
The product of two ample numbers is ample.
\end{corollary}
\emph{Proof.}
This follows immediately from Theorem \ref{aproductrule}, which states $a(l n) \geq a(l) \, a(n)$.
A number is ample if $a(n) > n$. So if $a(l) > l$ and $a(n) > n$, then $a(l n) > l n$. \qed
\begin{table}[b!]
\begin{small}
\[\arraycolsep=14pt
\begin{array}{lllll}
2^2\cdot 3 & 2^2\cdot 3\cdot 5^2 & 2^5\cdot 3\cdot 7 & 2^4\cdot 3\cdot 5^2 & 2^5\cdot 3\cdot 19 \\
 2^3\cdot 3 & 2^6\cdot 5 & 2^4\cdot 3^2\cdot 5 & 2^5\cdot 3\cdot 13 & 2^3\cdot 3\cdot 7\cdot 11 \\
 2^2\cdot 3^2 & 2^2\cdot 3^4 & 2^2\cdot 3^3\cdot 7 & 2^2\cdot 3^2\cdot 5\cdot 7 & 2^4\cdot 3^2\cdot 13 \\
 2^4\cdot 3 & 2^4\cdot 3\cdot 7 & 2^8\cdot 3 & 2^8\cdot 5 & 2\cdot 3^3\cdot 5\cdot 7 \\
 2^2\cdot 3\cdot 5 & 2^3\cdot 3^2\cdot 5 & 2^3\cdot 3^2\cdot 11 & 2^4\cdot 3^4 & 2^7\cdot 3\cdot 5 \\
 2^3\cdot 3^2 & 2^7\cdot 3 & 2^5\cdot 5^2 & 2^3\cdot 3\cdot 5\cdot 11 & 2^3\cdot 3^5 \\
 2^4\cdot 5 & 2^4\cdot 5^2 & 2^3\cdot 3\cdot 5\cdot 7 & 2^6\cdot 3\cdot 7 & 2^2\cdot 3^2\cdot 5\cdot 11 \\
 2^2\cdot 3\cdot 7 & 2^2\cdot 3\cdot 5\cdot 7 & 2^5\cdot 3^3 & 2^5\cdot 3^2\cdot 5 & 2^5\cdot 3^2\cdot 7 \\
 2^5\cdot 3 & 2^4\cdot 3^3 & 2^7\cdot 7 & 2^3\cdot 3^3\cdot 7 & 2^3\cdot 3\cdot 5\cdot 17 \\
 2^2\cdot 3^3 & 2^6\cdot 7 & 2^2\cdot 3^2\cdot 5^2 & 2^9\cdot 3 & 2^2\cdot 3\cdot 5^2\cdot 7 \\
 2^3\cdot 3\cdot 5 & 2^5\cdot 3\cdot 5 & 2^3\cdot 3^2\cdot 13 & 2^3\cdot 3\cdot 5\cdot 13 & 2^6\cdot 3\cdot 11 \\
 2^4\cdot 3^2 & 2^3\cdot 3^2\cdot 7 & 2^6\cdot 3\cdot 5 & 2^4\cdot 3^2\cdot 11 & 2^4\cdot 3^3\cdot 5 \\
 2^5\cdot 5 & 2^4\cdot 3\cdot 11 & 2^2\cdot 3^5 & 2^6\cdot 5^2 & 2^6\cdot 5\cdot 7 \\
 2^3\cdot 3\cdot 7 & 2^2\cdot 3^3\cdot 5 & 2^4\cdot 3^2\cdot 7 & 2^2\cdot 3^4\cdot 5 & 2^2\cdot 3^4\cdot 7 \\
 2^2\cdot 3^2\cdot 5 & 2^4\cdot 5\cdot 7 & 2^5\cdot 3\cdot 11 & 2^5\cdot 3\cdot 17 & 2^8\cdot 3^2 \\
 2^6\cdot 3 & 2^6\cdot 3^2 & 2^3\cdot 3^3\cdot 5 & 2^4\cdot 3\cdot 5\cdot 7 & 2^2\cdot 3^2\cdot 5\cdot 13 \\
 2^3\cdot 3^3 & 2^3\cdot 3\cdot 5^2 & 2^5\cdot 5\cdot 7 & 2^6\cdot 3^3 & 2^4\cdot 3\cdot 7^2 \\
 2^4\cdot 3\cdot 5 & 2^4\cdot 3\cdot 13 & 2^7\cdot 3^2 & 2^5\cdot 5\cdot 11 & 2^3\cdot 3^3\cdot 11 \\
 2^2\cdot 3^2\cdot 7 & 2^7\cdot 5 & 2^3\cdot 3\cdot 7^2 & 2^8\cdot 7 & 2^5\cdot 3\cdot 5^2 \\
 2^5\cdot 3^2 & 2^3\cdot 3^4 & 2^2\cdot 3^3\cdot 11 & 2^3\cdot 3^2\cdot 5^2 & 2^4\cdot 3^2\cdot 17 
\end{array}
\]
\\ \vspace{8pt}
\caption{ 
The first 100 recursively abundant numbers, which I call ample numbers.
A number $n$ is ample if $a(n) > n$, where $a(n)$ is the number of recursive divisors.
A concise Mathematica algorithm for the ample numbers is as follows: 
{\tt 
n = 2;
max = 1000; 
a = \{1\}; 
While[n <= max, a = Append[a, 1 + Total[Part[a, Delete[Divisors[n], -1]]]]; n++];
Select[Range[max], a[[\#]] > \# \&]
}
}
\end{small}
\label{ampletable}
\end{table}
\begin{lemma}
	No deficient numbers are ample.
\end{lemma}
\emph{Proof.}
Let $b(i)$ be the $i$th deficient number.
The proof is by induction on $i$.
First note that no deficient numbers are ample for $i<2$.
Assume no deficient numbers are ample up to but not including the $i$th one.
From Definition \ref{adef},
\begin{equation*}
	a(b(i)) 	=	  	1 + \sum_{m\propdiv b(i)} a(m).	
\end{equation*}
It is well known that all proper divisors of deficient numbers are deficient.
Then since, by assumption, no deficient numbers are ample up to the $i$th one, $a(m) \leq m$, and
\begin{equation*}
 	a(b(i)) 	\leq	  	1 + \sum_{m\propdiv b(i)} m.		
\end{equation*}
Since $b(i)$ is deficient, the sum of its proper divisors above is less than $b(i)$, and so
\begin{equation*}
	a(b(i))	\leq	  b(i),
\end{equation*}
that is, the $i$th deficient number is not ample.
This completes the inductive step.		\qed
\begin{lemma}
	No even perfect numbers are ample.
\end{lemma}
\emph{Proof.}
All even perfect numbers are of the form $2^{p-1} (2^p-1)$, where $2^p-1$ are prime.
From (\ref{aexpcd}), $a(2^c q) = 2^c(2+c)$, so $a(2^{p-1} (2^p-1)) = 2^{p-1} (p + 1)$,		
and the condition that an even perfect number is ample is $p+2 > 2^p$, which is never satisfied for $p \geq 2$.
Since the first perfect number occurs at $p=2$, no even perfect number is ample. 	\qed
\begin{theorem}
	All ample numbers are abundant or odd perfect numbers (if they exist).
\end{theorem}
\emph{Proof.}
This follows from Lemmas 1 and 2: if no deficient or even perfect numbers are ample, then only the abundant and odd perfect numbers can be ample. \qed
\subsection{Odd ample numbers} 
Like with the abundant numbers, there are odd ample numbers.
At first the opposite seemed true.
I thought there would be no odd ample numbers because of the special role of 2 in closed form expressions of $a(n)$; see (\ref{aexpc}), (\ref{aexpcd}) and (\ref{aexpcde}), for example.
To my surprise, I found that odd ample numbers do exist and, prompted by the analogy with abundant numbers, also found such numbers not divisible by 3.
The smallest ample numbers of each type are 
\begin{eqnarray*}
4.3 \times 10^{11}	&\simeq&	3^9 \cdot 5^5 \cdot 7^2 \cdot 11 \cdot 13 				 \\
3.3 \times 10^{81}	&\simeq&	5^{22} \cdot 7^{13} \cdot 11^8 \cdot 13^6 \cdot 17^5 \cdot 19^4 \cdot 23^3 \cdot 29^2 \cdot 31^2 \cdot 37^2 \cdot 41 \cdot \ldots \cdot 73,
\end{eqnarray*}
which are considerably larger than their abundant counterparts shown above.
This leads to the following conjecture:
\begin{conjecture}
	There exist ample numbers not divisible by the first $k$ primes for all $k$.
\end{conjecture}
\noindent
Note that if there is one ample number not divisible by the first $k$ primes, then there is an infinite number, since the product of two ample numbers is ample.
\begin{figure}[b!]
\includegraphics[width=0.85\textwidth]{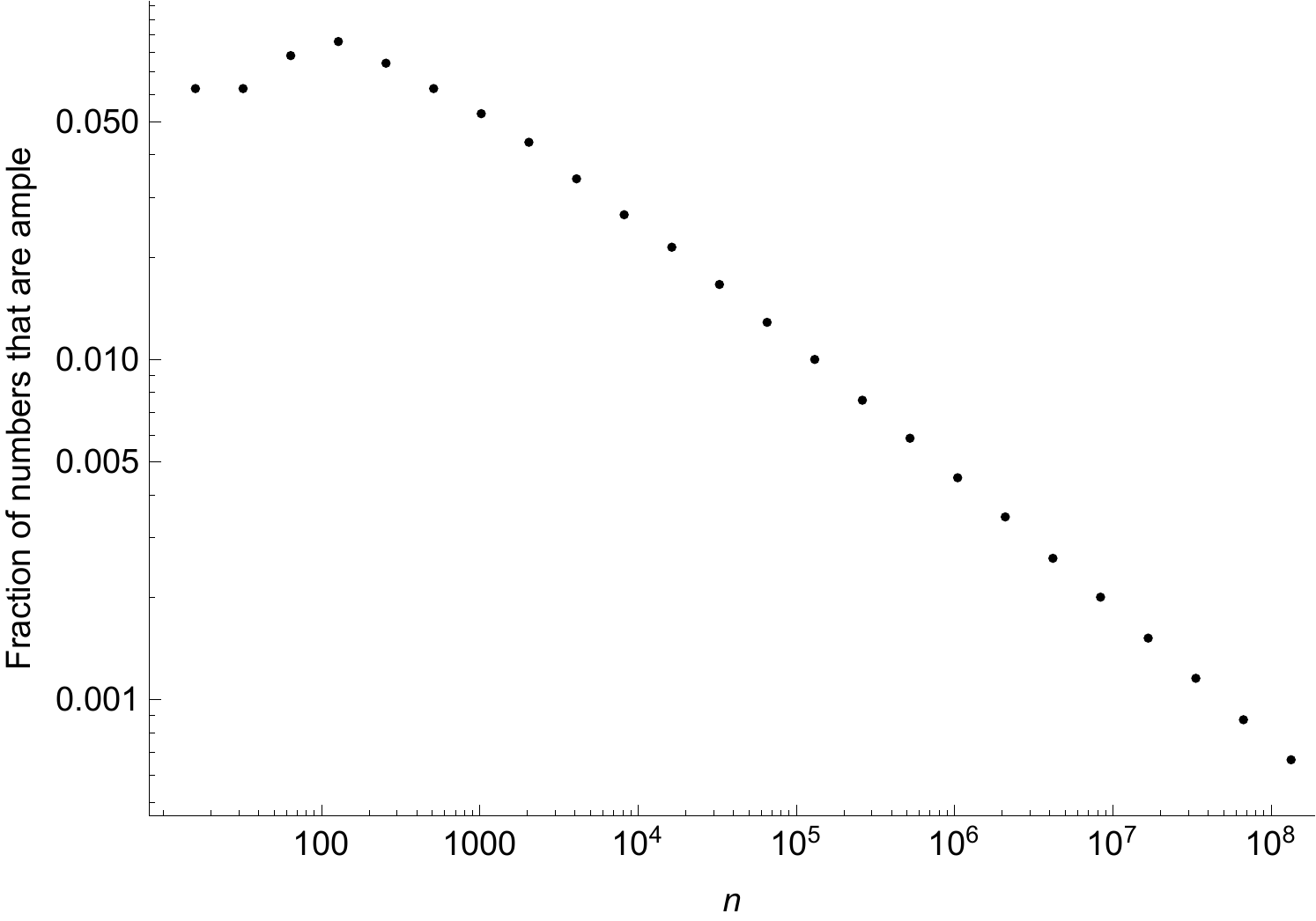}
\caption{\small
The fraction of numbers up to and including $n$ that are ample.
Unlike abundant numbers, which have natural density is between 0.2474 and 0.2480, the density of ample numbers decreases with $n$.
}
\vspace{-14pt}
\label{avalsplot}
\end{figure}
\section{Recursively perfect numbers}
\noindent
In this section I review perfect numbers and introduce recursively perfect numbers, which I call pristine numbers.
\subsection{Perfect numbers}
A number is perfect if it equals the sum of its proper divisors, that is,
\begin{equation*}
\sigma(n) - n = \sum_{m \divprop n} m = n.
\end{equation*}
The first several perfect numbers are $6, 28, 496, 8128, 33550336$.
Euclid showed that numbers of the form $2^{p-1}(2^p-1)$ are perfect for $2^p-1$ prime.
All known perfect numbers are even, and if an odd perfect number exists, it must be greater than $10^{1500}$ \cite{Ochem}.
Euler proved that all even perfect numbers are of the form given by Euclid, so there is a one-to-one correspondence between even perfect numbers and Mersenne primes. 
But it is not known if there are infinitely many of either.
\subsection{Pristine numbers}
Recursively perfect numbers are the recursive analog of perfect numbers.
I call them pristine numbers.
\begin{definition}
	A number $n$ is pristine if $a(n) = n$, where $a(n)$ is the number of recursive divisors.
	\label{pristinedef}
\end{definition}
\noindent
For example, 40 is pristine because $a(40) = 40$.
The first several pristine numbers are 1, 2, 4, 6, 8, 16, 32, 40, and the first 100 are shown in Table 2.
\begin{theorem}
There are no odd pristine numbers apart from 1.
\end{theorem}
\begin{proof}
From Corollary 1 in \cite{Fink}, $2^{\tau}$ divides $a(n)$, where $\tau$ is the maximum exponent in the prime factorization of $n$. 
\end{proof}
\begin{table}[b!]
\begin{small}
\[\arraycolsep=13pt
\begin{array}{llll}
1 & 2^7\cdot 3\cdot 13\cdot 31 & 2^{33} & 2^{49} \\
 2 & 2^{18} & 2^{29}\cdot 31 & 2^{25}\cdot 3\cdot 7\cdot 29\cdot 127\cdot 349 \\
 2^2 & 2^{19} & 2^{34} & 2^{39}\cdot 3\cdot 587 \\
 2\cdot 3 & 2^{15}\cdot 17 & 2^{25}\cdot 11\cdot 71 & 2^{50} \\
 2^3 & 2^{20} & 2^{35} & 2^{45}\cdot 47 \\
 2^4 & 2^{13}\cdot 11\cdot 23 & 2^{36} & 2^{51} \\
 2^5 & 2^{21} & 2^{37} & 2^{52} \\
 2^3\cdot 5 & 2^{17}\cdot 19 & 2^{19}\cdot 7\cdot 23\cdot 37\cdot 53 & 2^{53} \\
 2^6 & 2^{22} & 2^{38} & 2^{37}\cdot 3\cdot 7\cdot 3259 \\
 2^7 & 2^{15}\cdot 13^2 & 2^{39} & 2^{54} \\
 2^5\cdot 7 & 2^{23} & 2^{40} & 2^{55} \\
 2^8 & 2^{15}\cdot 3\cdot 107 & 2^{35}\cdot 37 & 2^{39}\cdot 13\cdot 59\cdot 103 \\
 2^3\cdot 3\cdot 11 & 2^{24} & 2^{41} & 2^{56} \\
 2^9 & 2^{25} & 2^{42} & 2^{51}\cdot 53 \\
 2^{10} & 2^{21}\cdot 23 & 2^{23}\cdot 3^4\cdot 11807 & 2^{57} \\
 2^{11} & 2^{26} & 2^{43} & 2^{58} \\
 2^{12} & 2^{27} & 2^{33}\cdot 3\cdot 431 & 2^{47}\cdot 11\cdot 227 \\
 2^9\cdot 11 & 2^{19}\cdot 13\cdot 37 & 2^{44} & 2^{59} \\
 2^{13} & 2^{28} & 2^{39}\cdot 41 & 2^{43}\cdot 3\cdot 29\cdot 1187 \\
 2^{14} & 2^{29} & 2^{45} & 2^{60} \\
 2^{11}\cdot 13 & 2^{30} & 2^{35}\cdot 11\cdot 131 & 2^{49}\cdot 37\cdot 73 \\
 2^{15} & 2^{21}\cdot 3\cdot 191 & 2^{46} & 2^{61} \\
 2^{16} & 2^{31} & 2^{41}\cdot 43 & 2^{62} \\
 2^9\cdot 3\cdot 47 & 2^{27}\cdot 29 & 2^{47} & 2^{51}\cdot 3\cdot 971 \\
 2^{17} & 2^{32} & 2^{48} & 2^{57}\cdot 59
\end{array}
\]
\\ \vspace{8pt}
\caption{
The first 100 recursively perfect numbers, which I call pristine numbers.
A number $n$ is pristine if $a(n) = n$, where $a(n)$ is the number of recursive divisors.
}
\end{small}
\label{PristineTable}
\end{table}
\begin{theorem}
Let $q, r$ and $s$ be odd primes.
All numbers of the following forms are pristine:
\begin{eqnarray*}
	2^c 																					,	\\
	2^c q 			&{\rm where}& 	q 				= 		c + 2								,	\\
	2^c q \, r , 			&{\rm where}& 	q \,	r 			= 		c^2 + 6c + 6						,	\\
	2^c q^2 , 			&{\rm where}& 	q^2	 			= 		(c^2 + 7c + 8) / 2!					,	\\
	2^c q \, r \, s, 		&{\rm where}& 	q \, 	r \, 	s 		= 		c^3 + 12 c^2 + 36 c + 26				,	\\
	2^c q^2 r, 			&{\rm where}& 	q^2 	r			= 		(c^3 + 13 c^2 + 42 c + 32) / 2!				.
\end{eqnarray*}
\end{theorem}
\emph{Proof.}
Using (\ref{aexpc}), (\ref{aexpcd}) and (\ref{aexpcde}), $a(n)$ can be written explicitly for simple forms of $n$:
\begin{eqnarray*}
	a(p^c) 				&=& 		2^c 									,	\\
	a(p^c q) 				&=& 		2^c 	(c + 2)							,	\\
	a(p^c q \, r) 			&=& 		2^c 	(c^2 + 6c + 6)						,	\\
	a(p^c q^2) 			&=& 		2^c 	(c^2 + 7c + 8) / 2!					,	\\
	a(p^c q \, r \, s) 			&=& 		2^c 	(c^3 + 12 c^2 + 36 c + 26)				,	\\
	a(p^c q^2 r) 			&=& 		2^c 	(c^3 + 13 c^2 + 42 c + 32) / 2!,				
\end{eqnarray*}
where for the case of $a(p^c q \, r \, s)$ some manipulation using the definition of $a(n)$ is required.
For $n$ to be pristine, $a(n) = n$.
Setting $p=2$, the theorem follows. \qed
\\ \indent
The first several values of pristine numbers of these forms are shown in Table 3.
\begin{theorem}
Let $q$ be prime.
No numbers of the form $2^c q^3$ or $2^c q^5$ can be pristine. 
\end{theorem}
\emph{Proof.}
For $n$ to be pristine, $a(n) = n$.
From (2),
$2^c q^d$ is pristine when $c$, $d$ and prime $q$ satisfy
\begin{equation}
	q^d 		= 		\sum_{i=0}^d \binom{d}{i} \binom{c+i}{i}.		
	\label{dqsum}			
\end{equation}
\begin{table}[b!]
\begin{small}
\begin{tabular*}{\textwidth}{@{\extracolsep{\fill}}llll}
$2^c$		& $2^c \cdot q$ 				& $2^c \cdot q \cdot r$  				& $2^c 	\cdot q \cdot r \cdot s$  			  	\\
 $2$ 			& $2 \cdot 3$ 					& $2^3 \cdot 3 \cdot 11$ 			 	& $2^7 	\cdot 3 \cdot 13 \cdot 31$				\\
 $2^2$ 		& $2^3 \cdot 5$  				& $2^9 \cdot 3 \cdot 47$				& $2^{37}	\cdot 3 \cdot 7 \cdot 3259$			\\
 $2^3$ 		& $2^5 \cdot 7$  				& $2^{13} \cdot 11 \cdot 23$			& $2^{39}	\cdot 13 \cdot 59 \cdot 103$			\\
 $2^4$ 		& $2^9 \cdot 11$  				& $2^{15} \cdot 3 \cdot 107$		 	& $2^{43} 	\cdot 3 \cdot 29 \cdot 1187$			\\
 $2^5$ 		& $2^{11} \cdot 13$  				& $2^{19} \cdot 13 \cdot 37$			& $2^{57}	\cdot 11 \cdot 67 \cdot 307$			\\
 $2^6$ 		& $2^{15} \cdot 17$  				& $2^{21} \cdot 3 \cdot 191$			& $2^{59}	\cdot 13 \cdot 127 \cdot 151$			\\
 $2^7$ 		& $2^{17} \cdot 19$  				& $2^{25} \cdot 11 \cdot 71$ 			& $2^{61} \cdot 3 \cdot 5 \cdot 18257$			\\
 $2^8$ 		& $2^{21} \cdot 23$  				& $2^{33} \cdot 3 \cdot 431$ 			& $2^{67} \cdot 3 \cdot 41 \cdot 2903$			\\
 $2^9$ 		& $2^{27} \cdot 29$  				& $2^{35} \cdot 11 \cdot 131$ 			& $2^{69} \cdot 19 \cdot 31 \cdot 659$			\\
 $2^{10}$ 		& $2^{29} \cdot 31$  				& $2^{39} \cdot 3 \cdot 587$			& $2^{71}	\cdot 5 \cdot 269 \cdot 313$			\\
 $2^{11}$ 		& $2^{35} \cdot 37$  				& $2^{47} \cdot 11 \cdot 227$			& $2^{73} \cdot 3 \cdot 29 \cdot 5237$ 			\\
 $2^{12}$ 		& $2^{39} \cdot 41$  				& $2^{49} \cdot 37 \cdot 73$			& $2^{75} \cdot 29 \cdot 71 \cdot 239$ 			\\
 \\
& $2^c \cdot q^2$  					& $2^c \cdot q^2 \cdot r$  					& $2^c 	\cdot q^3 	$		\\ 
& $2^{15} \cdot 13^2$ 		 		& $2^{167} \cdot 11^2 \cdot 20773$ 			& {\rm Numbers of this}		\\
& $2^{63} \cdot 47^2$		 		& $2^{419} \cdot 11^2 \cdot 313471$ 		& {\rm form cannot be}		\\
& $2^{623} \cdot 443^2$ 				& $2^{2587} \cdot 11^2 \cdot 71904083$ 		& {\rm pristine}				\\
& $2^{2255} \cdot 1597^2$	 		& $2^{2879} \cdot 47^2 \cdot 5425729$		& 						\\
& $2^{76719} \cdot 54251^2$			& $2^{3031} \cdot 11^2 \cdot 115558829$		& 						\\
& $2^{722975} \cdot 511223^2$		& $2^{3999} \cdot 11^2 \cdot 265124281$		& 	
\end{tabular*}
\vspace{0.1in}
\caption{
The first several pristine numbers for various forms of $n$.
A number $n$ is pristine if $a(n) = n$, where $a(n)$ is the number of recursive divisors.
}
\end{small}
\label{PristineSortedTable}
\end{table}
\\ \indent
For $d=3$, the right side of (\ref{dqsum}) is a cubic polynomial in $c$ that factors:
\begin{equation*}
q^3 = (c+4) (c^2 + 11 c + 12) / 3!.
\end{equation*}
But we do not know which part of 3! divides the linear and quadratic factors.
Let us consider separately the cases for when $c$ modulo 3 is equal to 0, 1 and 2: 
$$
q^3 =
\begin{cases}
(3 v+4) 	\cdot (3 v^2+11 v+4) / 2!	& \text{for } c = 3 v,		\\
(3 v+5)	\cdot (3 v^2+13 v+8)	/ 2!	& \text{for } c = 3 v + 1,	\\
(v+2) 	\cdot (9 v^2+45 v+38)/2!	& \text{for } c = 3 v + 2.	
\end{cases}
$$
Now the linear and quadratic factors are both integers.
For $2^c p^3$ to be pristine, the linear and quadratic factors must equal $q$ and $q^2$.
But the quadratic factors are not squares of the linear factors, so this is not possible in general.
It remains to check whether the quadratic and the square of the linear terms intersect at positive integers, which can happen at most twice.
By inspection, they do not. 
\\ \indent
For $d=5$, the right side of (\ref{dqsum}) is a polynomial of degree 5 in $c$ that factors:
\begin{equation*}
q^5 = (c+6) (c^4+34 c^3+331 c^2+914 c+640)/5!.
\end{equation*}
By similar arguments, not included here, $2^c p^5$ cannot be pristine.  	\qed
\begin{theorem}
Let $q$ be prime.
At most $d-1$ numbers of the form $2^c q^d$ are pristine for $d$ odd. 
\end{theorem}
\emph{Proof.}
It may well be that there are no pristine numbers of the form $2^c q^d$ for $d$ odd.
But I can only prove the weaker result.
Using the falling factorial notation, where $(c)_{i} = c (c-1)\ldots (c-i+1)$, (\ref{dqsum}) can be rewritten as
\begin{equation}
	q^d		= 		\frac{1}{d!} \sum_{i=0}^d (d-i)! \binom{d}{i}^2 (c+i)_{i}.
	\label{TA}
\end{equation}
Making use of the identity
\begin{equation*}
	\sum_{i=0}^d (-1)^i \binom{d}{i}^2 = 
	\begin{cases}
	0,			 						& d \quad {\rm odd},						\\
	(-1)^{(d/2)} \binom{d}{d/2}, 				& d \quad {\rm even},
	\end{cases}
\end{equation*}
for $d$ odd (\ref{TA}) can be expressed as 
\begin{eqnarray*}		
	q^d		&=& 		\frac{c+d+1}{d!} \sum_{i=1}^d (d-i)!   (c+i-1)_{i-1}	\sum_{j=i}^d (-1)^{j-i}	 	\binom{d}{j}^2		\\
	 		&=& 		\frac{c+d+1}{d!} P(c),					
\end{eqnarray*}
where $P(c)$ is a polynomial in $c$.
Let $d! = \alpha \beta$, where $\alpha$ is the largest proper divisor of $c+d+1$.
For $2^c q^d$ to be pristine, it is necessary that $q = (c+d+1)/\alpha$, and therefore
\begin{equation*}
\frac{(c+d+1)^{d-1}}{\alpha^{d-1}} = \frac{P(c)}{\beta}.
\end{equation*}
The numerators on both sides are polynomials in $c$ of degree $d-1$ in which the coefficient for the leading term is one.
For the equality to hold in general, we need $\alpha^{d-1} = \beta$, or $\alpha = (d!)^{1/d}$. 
But since $(d!)^{1/d}$ is not an integer for $d>1$, it does not hold in general.
Since two different polynomials (now the fractions, not just the numerators) of degree $d-1$ can match at at most $d-1$ places, there can be at most $d-1$ pristine numbers of the form $2^c q^d$ for odd $d$. \qed
\section{Open questions}
\noindent
Numbers can be sorted into three pots depending on whether they exceed, equal or are less than the sum of their proper divisors.
In a similar way, they can also be sorted on whether they exceed, equal or are less than the number of their recursive divisors.
The properties and structure of these two sets of pots have some intriguing parallels.
This correspondence suggests the possibility of deeper connections between the divisor function and its recursive analog, which I hope others might pursue.
\subsection{Open questions}
There are several open questions on this topic, and I list six here.
The first three concern ample numbers and the last three concern pristine numbers.
\begin{question}
	What is the asymptotic density of the ample numbers for large $n$?
\end{question}
\begin{question}
	I computed the least ample number that is odd and the least ample number not divisible by the first two primes.
	Do there exist ample numbers not divisible by the first $k$ primes, for all $k$?
\end{question}
\begin{question}
	If so, what is an efficient recipe for generating the smallest number not divisible by the first $k$ primes?
\end{question}
\begin{question}
	What is the asymptotic density of the pristine numbers for large $n$?
\end{question}
\begin{question}
	There are no pristine numbers of the form $2^c p^3$ and $2^c p^5$. 
	Are there none for $2^c p^d$, for odd $d$?
	I have not found any, and have shown that there are at most $d-1$.
\end{question}
\begin{question}
	Are there no pristine numbers of the form $2^c p^d$ for even $d>2$? 
	I have not found any.
\end{question}
I acknowledge Andriy Fedosyeyev for assistance with the search for odd ample numbers and for creating the divisor tree generator, lims.ac.uk/recursively-divisible-numbers. 

\end{document}